\documentclass[a4paper,11pt]{amsart}
\usepackage[plainpages=false]{hyperref}
\usepackage{amsfonts,latexsym,rawfonts,amsmath,amssymb,amsthm,mathrsfs}
\usepackage{amsmath,amssymb,amsfonts,latexsym,lscape,rawfonts}

\newtheorem{thm}{Theorem}[section]
\newtheorem{claim}[thm]{Claim}
\newtheorem{cor}[thm]{Corollary}
\newtheorem{lemma}[thm]{Lemma}

\newtheorem{prop}[thm]{Proposition}
\newtheorem{rmk}[thm]{Remark}

\title{The interior regularity of the Calabi flow on a toric surface}
\author{Xiuxiong Chen, Hongnian Huang  and Li Sheng}

\thanks{The second named author is financially supported by the FMJH (Fondation math\'ematique Jacques Hadamard).}

\date{11/25/2012}

\begin{document}

\maketitle

\begin{abstract}
Let $X$ be a toric surface with Delzant polygon $P$ and $u(t)$ be a solution of the Calabi flow equation on $P$. Suppose the Calabi flow exists in $[0, T)$. By studying local estimates of the Riemann curvature and the geodesic distance under the Calabi flow, we prove a uniform interior estimate of $u(t)$ for $t < T$.
\end{abstract}
\tableofcontents
\section{Introduction}
 According to Calabi \cite{Ca1, Ca2},  a K\"ahler metric is called extremal if the gradient vector field of its scalar curvature is a holomorphic vector field.  When this vector field vanishes, it is a constant scalar curvature K\"ahler (cscK) metric.  One of the core problems in K\"ahler geometry is to establish the existence of cscK metrics with appropriate algebraic conditions (c.f. Yau-Tian-Donaldson conjecture). It is now known that the existence of cscK metrics implies $K$-stability on a polarized K\"ahler manifold $X$ \cite{Stoppa, Ma2, D5, CT}. But the existence problem for general cscK metrics in a non-canonical class is largely open.\\
  
 In 1982, Calabi proposed to deform any K\"ahler potential in the direction of its scalar curvature:
\begin{equation}
\frac{\partial \varphi}{\partial t} = R_\varphi - \underline{R}, \label{CalabiflowComplex}
\end{equation}
where $\underline{R}$ is the average of $R_\varphi$.  This is a parabolic flow of 4th order and one
can establish short time existence for smooth initial data \cite{ChenHe}. The first named author
conjectured  that the Calabi flow exists for all time (c.f. \cite{ChenHe3}). In a joint work with W. He, the first named author proved that the Calabi flow can be extended indefinitely if the $Ric$ curvature is bounded \cite{ChenHe}.  More recently, joint with R. Feng, the second named author proved the global existence of the Calabi flow in a special  K\"ahler manifold $X=\mathbb{C}^2/(\mathbb{Z}^2 + i \mathbb{Z}^2)$ \cite{FH}. ~ Since the work of \cite{ChenHe}, there are extensive work in the subject of the Calabi flow, c.f., \cite{TW, Fi, He, H1, HZ, St, H3, Sz} etc.\\

In a remarkable series of work \cite{D1, D2, D3, D4}, Donaldson proved the existence of constant scalar curvature K\"ahler metric in a polarized $K$-stable toric surface. In \cite{ZZ}, X. Zhu and B. Zhou studied the weak solution of the Abreu's equation. In \cite{CLS}, B. Chen, A.-M. Li and the third named author extended Donaldson's work to the case of extremal K\"ahler metric in toric surface. \\

Inspired by Donaldson's program on toric surfaces,  the first and second named author studied the long time existence of the Calabi flow \cite{CH} in toric K\"ahler surface. In our current paper,  let $X$ be a toric surface with Delzant polygon $P$ and $u$ be a sympletic potential in $P.\;$ By Abreu's work, we can reduce 
the scalar curvature equation to
\begin{equation}
A_u =  - U^{ij} \left(\frac{1}{\det(D^2 u)} \right)_{ij} \label{cscKSymp}
\end{equation}
where $U^{ij}$ is the cofactor matrix of $u_{ij} .\;$  Following the tradition, we denote the average of scalar curvature $\underline{R}$ as $\underline{A}$ in toric settings.  Suppose that $u(t)$ is a one parameter family of symplectic potentials satisfying the Calabi flow equation.  Then,\begin{equation}
{{\partial u}\over {\partial t}}  = \underline{A} + U^{ij} \left(\frac{1}{\det(D^2 u)} \right)_{ij}.
\label{CalabiflowSymp}
\end{equation}
  
Suppose the Calabi flow exists in the time interval $[0,T)$. We then have the following theorem. 
\begin{thm}
\label{main}
For any constant $\epsilon_0 > 0$, let  $P_{\epsilon_0}$ be the subset containing all the points in $P$ whose Euclidean distance to $\partial P$ is greater than $\epsilon_0$. Then there exist constant $C(\epsilon_0)$ and $ C(k, \epsilon_0)$ for all $k \in \mathbb{Z}^+$ such that for any $t < T$ and any point $x \in  P_{\epsilon_0}$, we have
$$
 (D^2 u)(t, x) > C(\epsilon_0),  \quad
|u|_{C^k} < C(k, \epsilon_0).
$$
\end{thm}

This is a parabolic version of interior estimates of Donaldson \cite{D2} where he proved the same estimates for constant scalar curvature K\"ahler toric metric.  In Donaldson's proof,
it is crucial that the scalar curvature is uniform bounded.  Viewing equation (\ref{cscKSymp}) as a second order elliptic equation on ${1\over {\det(D^2 u)}},\;$ Donaldson controlled  the lower bound of  $\det (D^2 u)$ in $P_{\epsilon_0}$ by maximum principle via some cleverly constructed barrier functions.  He obtained the upper bound by adopting the techniques of Trudinger and Wang \cite{TrWa, TrWa1} which also relies on the construction of the barrier functions and the maximum principle. Then appealing to the real linearized Monge-Amp\`ere theory \cite{CF, CG} and Schauder's estimates, he obtained 
$$
 (D^2 u)(x) > C(\epsilon_0),  \quad
|u|_{C^k} < C(k, \epsilon_0).
$$
In our case, however, we only have $L^2$ control of $A_t$ along the Calabi flow. A pointwise
 bound on the $\det(D^2 u)$ seems to be beyond reach due to the lack of maximum principle. Thus,  we are not able to apply Donaldson's method in our paper. A crucial step is to prove the following intermediate theorem.

\begin{thm}
\label{key}
For every small constant $\epsilon_0 > 0$, there exists $C(\epsilon_0) > 0$ such that for any $t < T$ and $x \in P_{\epsilon_0}$, we have
$$
Q(t, x) d_{u(t)}^2(x, \partial P_{\epsilon_0}) < C(\epsilon_0),
$$
where $Q(t, x) = (|Rm| + |\nabla Rm|^{2/3} + |\nabla^2 Rm|^{1/2}) (t,x).$
\end{thm}

One immediate remark is that,  even if this theorem holds, we do not have control on curvature in the interior, unless we can prove that the geometric distance from interior to the boundary is non-trivial.  As a first step in our paper, we prove the following:
 
\begin{thm}
\label{distance}
For any small $\epsilon_0 > 0$, there exists a constant $C(\epsilon_0)$ such that for any $t < T$, we have 
$$
d_{u(t)}(P_{\epsilon_0}, P_{2\epsilon_0}) > C(\epsilon_0).
$$
In other words, the interior of the polygon contains some geometric ball of fixed size during the flow at finite times.
\end{thm}

Now, we discuss the main ideas in this paper.\\
 
We use a  blow-up analysis method adopted from \cite{ChenHe2} \cite{ChenHe3}.  However, there are substantial new difficulties arising in our settings:\\

\begin{enumerate}
\item Unlike \cite{ChenHe2}, we do not have a uniform Sobolev constant bound here.  This
causes troubles in both ends: to obtain convergence of the re-scaling sequence, to classify the ``bubble"  in the limit as well as to 
control the higher derivatives (so that we can ``bootstrap" convergence to draw contradictions).  The lack of injectivity radius control can be overcomed, if
we know that the scalar curvature is a constant. Then,  we can study equation (\ref{cscKSymp})
in the tangential space and obtain  the higher derivatives of the curvature there.  This method is not available
to us since we only have integral estimates of of curvature functions.   To overcome the difficulty of collapsing, we needs to uniformly control $|\nabla Rm|$ after blowing up as in \cite{FH}. \\

\item In both \cite{ChenHe2} and \cite{ChenHe3}, they blow up at the  global maximum of the curvature function $|Rm|$ at each time slice.
Therefore, one has curvature control globally (over both forward time and space) after blowing up. 
For Ricci flow,  W. Shi \cite{Shi} can obtain higher derivatives control via maximum principle. 
With similar ideas in mind, but using integral estimate,  W. He and the first named author \cite{ChenHe2}  control
$$
\int_X |\nabla^k Rm|^2 ~ dg
$$
instead. Here curvature uniform bound is crucial in the long calculations.
Then the Sobolev constant bound enables them to obtain higher regularities control of $Rm$, i.e., $|\nabla^k Rm|$. \\

There are two immediate difficulties in estimating the evolution of the above integral in our setting. First, the geometrical quantity $Q(t, x)$ is not known to be bounded in any time slice $t<0$ after blowing up. Second (even in the absence of the first difficulty), it is hard to obtain point-wise estimates due to the lack of Sobolev constant control in our settings. Thus we need to construct a good cut off function defined only on the polygon (which really measures the
distance between any two  orbits by the toric action) and localize our  integral of curvatures:
$$
\int_X f |\nabla^k Rm|^2 ~ dg.
$$

Upon controlling this quantity,  by using M-condition and uniform curvature bound, we obtain that ( c.f. \cite{D3})
\[
(D^2 u)(x) < C,
\]
where $x$ is around the blowup point. The higher regularities follow from here. We remark that this is one of the main reasons why we have to restrict ourselves to an interior estimate at this stage. Here we adopt similar arguments  in the appendix of \cite{FH}  to give us the higher regularities control of $Rm$.\\

\item To deal with the cut off function $f$, some of the technical difficulties come from integration by parts, for example, $\int_X f ~ dg$ is not bounded. The more substantial difficulty is that we need to choose some parabolic box in the polygon in the sense of geodesic distance. Then, the issue of how the geodesic distance changes over time is crucial.  In fact, one of the main challenges in a geometric flow (Ricci flow, Calabi flow) is to compare the distance function in different time slices, for example $t=0$ and $t=-1$. Note that the evolution of metric is controlled by the Hessian of the scalar curvature which, at the present stage, is precisely what we hope to control.  \\

A key step is that we divide the geodesic segment which realizes the distance at time $t=-1$ into three sets:  the set where the concentration of
$|Rm|^2$ is large, the set where the concentration is controlled and the oscillation of metric, comparing to $t=0$, is large and the set where the oscillation is bounded and the concentration is bounded \footnote{Roughly speaking, a point contains a nontrivial concentration energy if there is certain line segment (in Euclidean sense, passing through this point) where the integration of $|Rm|^2$ is nontrivial over this line segment. }. For the first set, it is easy to show that its Lebesgue measure is controlled. For the second set, we prove that its Lebesgue measure is also controlled because the evolution of metric is controlled by the Hessian of the scalar curvature, hence by the difference of the Calabi energy at $t=-1$ and $t=0$ which is bounded. Therefore, we are able to prove that, if the distance between $x$ and $\partial P_{\epsilon_0}$ is $L$ at $t=0$, then the distance between $x$ and $\partial P_{\epsilon_0}$ is at least $cL$ at $t=-1$ for some uniform constant $c$. \\

\end{enumerate}

Finally, to prove Theorem (\ref{main}) using Theorem (\ref{key}), one needs to show that  the geodesic distance and Euclidean distance are somehow equivalent in $P_{\epsilon_0}$. This is proved in Theorem (\ref{distance}). Notice that the length of a curve is
\begin{eqnarray}
\label{Def_distance}
\int_0^{s_0} \sqrt{D^2(u)(\gamma'(s), \gamma'(s))} ~ds.
\end{eqnarray} 
The key observation is that (\ref{Def_distance}) is bounded from below because
$$
\int_P Trace(u^{ij}(t)) < C
$$
uniformly. \\

With all the results obtained so far, one can control $Rm$ and its covariant derivatives in $P_{\epsilon_0}$. Then one can get Theorem (\ref{main}) using Krylov-Safonov and Schauder's estimates.
\\

\begin{rmk} After the celebrated work of G. Perelman \cite{P1},  it has now become a powerful tool in the Ricci flow to select a local maximum of certain geometry quantity (like curvature) and to apply careful analysis in local parabolic box (c.f. \cite{CW1, CW2}). While we clearly draw inspirations from Perelman's work,  it seems still a novelty to localize estimates in the Calabi flow. There are obvious, substantial new difficulties arising because of its higher order. Some of these new difficulties are fundamentals and require thorough new investigation and invention. However, the first named author suspects that it will be a constant theme in the study of geometry flow to deal with the oscillations of the distance functions over time parameter.  \\
\end{rmk}

{\bf Acknowledgment: } The second and third named author would like to express their gratitude to Professor Paul Gauduchon and Frank Pacard for their support. The second named author would like to thank Professor Pengfei Guan and Vestislav Apostolov for stimulating discussions. The third named author would like to thank Professor Anmin Li for his support. 

\section{Notations and Setup}
Let $X$ be a K\"ahler manifold with a complex structure $J$ and a K\"ahler class $[\omega]$. There is a one to one correspondence between the sets of all K\"ahler metrics and a set of relative K\"ahler potentials $\mathcal{H}_0$, where
$$
\mathcal{H}_0 = \{\varphi \in C^\infty(M) ~ | ~ \omega + i \partial \bar{\partial} \varphi > 0 \} / \mathbb{R}.
$$
The Calabi flow equation is 
$$
\frac{\partial \varphi}{\partial t} = R_\varphi - \underline{R},
$$
where $R_\varphi$ is the scalar curvature and $\underline{R}$ is its average. The Calabi flow decreases the Calabi energy which is
$$
\int_X (R_\varphi - \underline{R})^2 ~ d \omega_\varphi^n.
$$
The evolution equation of the bisectional curvature of the Calabi flow is
$$
\frac{\partial Rm}{\partial t} = - \triangle^2 Rm + \nabla^2 Rm * Rm + \nabla Rm * \nabla Rm.
$$
Suppose $X$ is a toric manifold with a Delzant polytope $P$. Then the toric invariant K\"ahler metric is one to one corresponding to the set of symplectic potentials $u$ satisfying the Guillemin boundary conditions up to an affine function. The Calabi flow equation in the sympletic side is
$$
\frac{\partial u}{\partial t} = \underline{A} - A_u,
$$
where $A_u = - \sum_{i, j} u^{ij}_{~ij}$ by Abreu's work and $\underline{A}$ is its average. The distance between any two symplectic potentials $u_1$ and $u_2$ is
$$
\sqrt{\int_P (u_1 - u_2)^2 ~ d\mu}.
$$

In the calculations of later sections, we also adopt the following notations:
\begin{itemize}

\item $| ~ \cdot ~ |_E$ : Euclidean metric.

\item $| ~ \cdot ~ |_u$ or $| ~\cdot ~|$ : Riemannian metric.

\item $D$: Euclidean derivative.

\item $\nabla$: Covariant derivative.

\end{itemize}

At the end of this section, we would like to introduce the $M$-condition from Donaldson's paper \cite{D3}. Let $l$ be a line interval of $P$ parameterizing by $p + s \nu, s \in [-3R, 3R]$, where $\nu$ is a unit vector. We say $u$ satisfying the $M$-condition on $l$ if
$$
| (D u (p- R \nu) - D u (p + R \nu) ) \cdot \nu |_E < M.
$$
We say $u$ satisfying the $M$-condition on $P$ if for any line interval $l \subset P$, $u$ satisfies the $M$-condition on $l$.

\section{Distance Control}
Our goal in this section is to prove Theorem  (\ref{distance}). Before going into the details of the proof, we would like to explain the ideas of the proof to our readers. Our first observation is Corollary (\ref{interior control}), i.e.,
$$
\int_P Trace(u^{ij}(t)) < C,
$$
where $C$ is a positive constant independent of $t$. This is a consequence of the facts that the Calabi flow decreases the Calabi energy and the geodesic distance in the space of relative K\"ahler potentials. Our second observation is Proposition (\ref{bad line}) where we show that if $Trace(u^{ij})$ is large at one point $x$, then we can show that along any interval $l$ with fixed diameter passing through $x$ in $P_{\epsilon_0}$, $\int_l Trace(u^{ij})$ is also large, provided that the $L^2$ norm of $Rm$ on $l$ is bounded. With these two observations and the fact that the Riemann length of a curve is
$$
\int_0^{s_0} \sqrt{u_{ij}(\gamma'(s), \gamma'(s))}~ ds,
$$
we can prove Theorem (\ref{distance}).

Now we proceed to give some lemmas in order to prove Corollary (\ref{interior control}). By Theorem 1.5 of Calabi and Chen \cite{CC}, the Calabi flow decreases the distance, we obtain that
$$
\int_P (u(0)-u(t/2))^2 ~ d \mu \geq \int_P (u(t/2)-u(t))^2 ~ d\mu.
$$
It implies that there exists a constant $C > 0$ such that for any $t < T$, we have
$$
\int_P u^2(t) ~ d\mu < C.
$$
Since the Calabi energy decreases under the Calabi flow, we have for any $t < T$,
$$
\int_P A^2(t) ~ d\mu < C.
$$
Thus we obtain that for any $t < T$,
$$
\left|\int_P u(t) A(t) ~ d \mu \right| < C.
$$
Also
$$
\left|\int_P u(0) A(t) ~ d \mu \right| < C.
$$
Then we have
$$
\left|\int_P (u(t) - u(0)) A(t) ~ d \mu \right| < C.
$$
Integration by parts as Lemma 3.3.5 in \cite{D1}, we have 
$$
\left|\int_P (u(t) - u(0))_{ij} u^{ij}(t) ~ d \mu \right| < C + 2 \left| \int_{\partial P} (u(t)-u(0)) ~ d \sigma \right| .
$$
Thus we have
$$
\left|\int_P u_{ij}(0) u^{ij}(t) ~ d \mu \right| < C + 2 \left| \int_{\partial P} (u(t)-u(0)) ~ d \sigma \right| .
$$
Similar calculations show
$$
\left|\int_P u_{ij}(t) u^{ij}(0) ~ d \mu \right| < C + 2 \left| \int_{\partial P} (u(t)-u(0)) ~ d \sigma \right| .
$$
We also observe the following lemma.
\begin{lemma}
There exists a constant $C$ such that for any $t < T$, we have
$$
\left| \int_{\partial P} (u(t)-u(0)) ~ d \sigma \right| < C.
$$
\end{lemma}

\begin{proof}
Notice that for any $t < T$, by the calculations of section 5 in \cite{H1}, we have
$$
\int_0^t \int_P u^{ij}(s) A_{ik}(s) A_{jl}(s) u^{kl} (s) ~ d \mu ds = \int_P A^2(0) - A^2(t) ~ d \mu.
$$
Thus we obtain
$$
\left| \int_0^t \int_P u^{ij}(s) A_{ij}(s) ~ d \mu dt \right| < C.
$$
Integration by parts, we get
$$
\left| \int_0^t \int_{\partial P} A(s) ~ d \sigma dt \right| < C.
$$
Since
$$
\int_{\partial P} (u(t)-u(0)) ~ d \sigma =\int_0^t \int_{\partial P} \underline{A} - A(s) ~ d \sigma dt ,
$$
we obtain the result.
\end{proof}
As a corollary, we have
\begin{cor}
\label{interior control}
There exists a constant $C$ such that for any $t < T$, we have
$$
\left|\int_P u_{ij}(t) u^{ij}(0) ~ d \mu \right| < C, \quad \left|\int_P u_{ij}(0) u^{ij}(t) ~ d \mu \right| < C.
$$
\end{cor}

To continue, we prove the following lemma first which is an extension of Lemma 4 in \cite{D3}.
\begin{lemma}
\label{Upper bound of Hessian}
Suppose $l$ is a line interval inside $P$ such that $\int_{l} |Rm|^2$ is bounded by a constant $C$. We parametrize $l = \{p+s \nu ~ : ~ -3R \leq s \leq 3R\}, $ where $p \in l$ is the middle point of $l$ and $\nu$ is a unit vector. Suppose $u$ satisfies the $M$-condition on $l$, then 
\begin{itemize}
\item $R>1$, we have
$$
u_{ij}(0)\nu^i\nu^j \leq e^{\frac{M-C}{2}}.
$$

\item $R \leq 1$,  we have

$$
u_{ij}(0) \nu^i\nu^j \leq  \frac{e^{M/2} -1}{CR}.
$$
\end{itemize}
\end{lemma}

\begin{proof}
We can suppose that $\nu$ is the unit vector in the $x_1$ direction and $p$ is the origin point. Let $H(s)=u_{11}(s,0).$
We apply the definition of the $M$-condition to obtain $$\int^R_{-R} H(s) ds \leq M.$$ Since
$$\frac{d^2}{ds^2}H(s)^{-1} \leq |Rm|(s),$$ 
for any $s_0 > 0$, we obtain
$$
\left(\frac{1}{H} \right)'(s_0) - \left(\frac{1}{H} \right)'(0) \leq \int_0^{s_0} |Rm|(s) ~ds \leq C \sqrt{s_0}.
$$
Suppose $H(0)^{-1}=\epsilon_0, \left(\frac{1}{H} \right)'(0) = \epsilon_1.$ Then 
\begin{eqnarray*}
H(s_0)^{-1} - H(0)^{-1} &=& \int_0^{s_0} \left(\frac{1}{H} \right)'(s) ~ ds \\
&\leq& \int_0^{s_0} C\sqrt{s} + \epsilon_1 ~ ds\\
&=& \epsilon_1 s_0 + C (s_0)^{3/2}.
\end{eqnarray*}
Thus $$H(s)+H(-s) \geq \frac{1}{\epsilon_0+ \epsilon_1 s +C s^{3/2}} + \frac{1}{\epsilon_0 - \epsilon_1 s +C s^{3/2}}
\geq \frac{2}{\epsilon_0 +C s^{3/2}}.$$ 

This gives 
\begin{eqnarray*}
M \geq \int_{-R}^R H(s) ~ ds \geq 2 \int^R_0 \frac{ds}{\epsilon_0 +C s^{3/2}}\\
\end{eqnarray*} 
If $R > 1$, then we have
$$
M \geq 2 \int^1_0 \frac{ds}{\epsilon_0 +C s} + C = C + 2\ln(\epsilon_0 + C) - 2 \ln \epsilon_0,
$$
which implies
$$
\epsilon_0 \geq e^{\frac{C-M}{2}}.
$$

If $R \leq 1$, then we have
$$
M \geq 2 \int^R_0 \frac{ds}{\epsilon_0 +C s}=  2 \ln \frac{C R + \epsilon_0}{\epsilon_0},
$$
which implies
$$
\epsilon_0 \geq \frac{CR}{e^{M/2} -1}.
$$
Thus we obtain the results.
\end{proof}

Now let $x$ be the axis along $l$ and $z$ be another axis. We have the following observation due to Donaldson \cite{D3}:
\begin{lemma}
\label{coordinate}
$$
|u^{zz}_{~xx}|(s) \leq |Rm|(s) u^{zz}(s) u_{xx}(s)
$$
and
$$
|u^{xx}_{~xx}|(s) \leq 2 |Rm|(s) u^{xx}(s) u_{xx}(s).
$$
\end{lemma}

\begin{proof}
Let $s$ be the origin point. Set 
$$v(x, z) = u(x + a z, z),$$ where $a$ is some constant to be determined. Then
\begin{eqnarray*}
(D^2 v) (x, z) = 
\left(
\begin{array}{cc}
1 & 0 \\
a & 1
\end{array}
\right)
(D^2 u) (x+ a z, z)
\left(
\begin{array}{cc}
1 & a \\
0 & 1
\end{array}
\right).
\end{eqnarray*}
It shows that
\begin{eqnarray*}
(D^2 u)^{-1} (x+az, z) = 
\left(
\begin{array}{cc}
1 & a \\
0 & 1
\end{array}
\right)
(D^2 v)^{-1} (x,z)
\left(
\begin{array}{cc}
1 & 0 \\
a & 1
\end{array}
\right).
\end{eqnarray*}
It means that $u_{xx} (x+az, z) = v_{xx}(x,z)$ and $u^{zz}(x+az, x) = v^{zz}(x,z)$. Thus we have
$$
u^{zz}_{~xx} (x+az, z) = v^{zz}_{~xx} (x,z).
$$
Since $v_{xz}(x,z) = a u_{xx}(x+az, z) + u_{xz}(x+az, z)$, we can choose an appropriate constant $a$ such that $v_{xz}(0, 0) = 0$. Thus calculations in Lemma 4.3 of \cite{H1} show that
\begin{eqnarray*}
& & |v^{zz}_{~xx}(0, 0) v_{zz}(0, 0) v^{xx}(0, 0)| \\
&\leq& \sqrt{\sum_{i,j,k,l} (v^{ij}_{~kl})^2(0,0) \frac{v^{kk}(0,0) v^{ll}(0,0)}{v_{ii}(0,0) v_{jj}(0,0)}}\\
& = & \sqrt{\sum_{i,j,k,l} v^{ij}_{~kl}(0,0) v^{kl}_{~ij}(0,0)} \\
& = & |Rm|(0, 0).
\end{eqnarray*}
Thus
$$
|u^{zz}_{~xx}(0, 0)| = |v^{zz}_{~xx}(0, 0)| \leq u^{zz}(0, 0) u_{xx}(0, 0) |Rm|(0, 0).
$$
Notice that
$$
u^{xx}(x+az, z) = v^{xx}(x,z) + 2a v^{xz}(x,z) + a^2 v^{zz}(x,z).
$$
Thus
$$
u^{xx}_{~xx}(x+az,z)=v^{xx}_{~xx}(x,z) + 2a v^{xz}_{~xx}(x,z) + a^2 v^{zz}_{~xx}(x,z).
$$
Notice that
$$
|v^{xx}_{~xx}|(0, 0) \leq |Rm|(0, 0)
$$
and
$$
|v^{xz}_{~xx}|(0, 0) \leq \sqrt{v^{zz}(0, 0) v_{xx}(0, 0)}|Rm|(0, 0).
$$
Thus we obtain
\begin{eqnarray*}
& &|u^{xx}_{~xx}|(0, 0)\\
&\leq& |Rm|(0, 0) ( 1 + 2a \sqrt{v^{zz}(0, 0) v_{xx}(0, 0)} + a^2  v^{zz}(0, 0) v_{xx}(0, 0))\\
&\leq& 2 |Rm|(0, 0) (1 +  a^2  v^{zz}(0, 0) v_{xx}(0, 0))\\
&\leq& 2 |Rm|(0, 0) (v^{xx}(0, 0)+  a^2  v^{zz}(0, 0)) v_{xx}(0, 0)\\
&\leq& 2 |Rm|(0, 0) u^{xx}(0, 0) u_{xx}(0, 0).
\end{eqnarray*}
\end{proof}

\begin{lemma}
\label{comparison}
Let $l$ be a line interval in $P$ parameterizing by $l =p + s \nu, s \in [-3R, 3R] $ and $\nu$ is a unit vector parallel to the $x$-axis. Suppose $\int_{-3R}^{3R} |Rm|^2(s) ~ ds < C$ and $u$ satisfies the $M$-condition on $l$. Then there exists $s_0 > 0$ depending on $M, C$ such that one of the following case occurs:
\begin{itemize}
\item For any $s \in [0,s_0]$, $$ u^{zz} (s) \geq u^{zz} (0) / 2.$$
\item For any $s \in [-s_0,0]$, $$u^{zz} (s) \geq u^{zz} (0) / 2.$$
\end{itemize}
Also we have either 
for any $s \in [0,s_0]$, $$ u^{xx} (s) \geq u^{xx} (0)/ 2.$$
or for any $s \in [-s_0,0]$, $$u^{xx} (s) \geq u^{xx} (0) / 2.$$
\end{lemma}

\begin{proof}
By Lemma (\ref{Upper bound of Hessian}), we know that $u_{xx}(s), s \in [-R, R]$ is bounded by a constant $C$. Thus we have for $s \in [-R, R]$, 
$$
-C |Rm|(s) u^{zz}(s) \leq u^{zz}_{~xx}(s) \leq C |Rm|(s) u^{zz}(s).
$$
Let $f(s) = \ln u^{zz}(s)$. We have
$$
f''(s) + f'^2(s) = \frac{u^{zz}_{~xx}}{u^{zz}}(s).
$$
Thus for $s \in [-R, R]$, we obtain
$$
|f'' + f'^2|  (s) \leq C |Rm|(s).
$$
Dividing both sides by $1+f'^2$, we obtain
$$
\left| \frac{f''}{1+f'^2} \right| \leq 1 + C |Rm|(s).
$$
Integrating both sides, we have
$$
|\arctan (f'(s)) - \arctan (f'(0))| \leq C \sqrt{|s|}
$$
Without loss of generality, we can assume $f'(0) \geq 0$. Then for $s \in [0, R]$, we have
$$
\arctan(f'(s)) \geq - C \sqrt{s}
$$
which is
$$
f'(s) \geq \tan (-C \sqrt{s}).
$$
Thus by choose $s_0$ appropriately, we conclude that for any $s \in [0, s_0]$,
$$
f(s) \geq f(0) - C s^{3/2}
$$
and 
$$
u^{zz} (s) \geq u^{zz}(0) / 2.
$$
Similar calculations show 
$$
u^{xx} (s) \geq u^{xx}(0) / 2
$$
for $s \in [0, s_0]$.
\end{proof}

Our discussions lead to the following proposition.
\begin{prop}
\label{bad line}
Suppose there exists a unit vector $\nu$ such that $$(D^2 u)(0)(\nu, \nu) < \epsilon,$$ then
$$
\int_{-R}^R trace(u^{ij}(s)) ~ ds \geq \frac{C}{\epsilon},
$$
where $C$ only depends on $M, R$ and $\int_{-R}^R |Rm|^2(s) ~ ds$.
\end{prop}
\begin{proof}
Since $(D^2 u)(0)(\nu, \nu) < \epsilon$, we conclude that the smallest eigenvalue of $(D^2 u)(0)$ must be less than $\epsilon$. Thus $Tr ((D^2 u)^{-1})(0) > 1/ \epsilon$. Without loss of generality, we can assume that $u^{zz}(0) > 1/ \epsilon$. Then our previous discussions show that there exists $s_0 > 0$ such that for all $s \in [0, s_0]$ (or for all $s \in [-s_0, 0]$),
$$
u^{zz} (s) \geq \frac{u^{zz}(0)}{2}.
$$

Then 
\begin{eqnarray*}
& &\int_{-R}^R trace(u^{ij}(s)) ~ ds \\
&\geq& \int_0^R u^{zz}(s) ~ ds \\
&\geq& \int_0^{\min \{s_0,R\}} \frac{1}{2 \epsilon } ~ ds\\
&=& \frac{\min \{s_0,R\}}{2 \epsilon}.
\end{eqnarray*}
Thus our conclusion holds.
\end{proof}

Now we give a proof of Theorem (\ref{distance}).
\begin{proof}[Proof of Theorem (\ref{distance}).]

Recall that for any $t < T$ we have 
$$
\int_P (u(t))^2 ~ d\mu < C.
$$
Then it is easy to see that there exists a constant $M > 0$ such that for any $t < T, x \in P_{\epsilon_0/2}$, we have
$$
|u(t, x)| < M, ~ |D u(t,x)| < M.
$$
It shows that for any $t < T$ and any line interval $l \subset P_{\epsilon_0/2}$, $u(t)$ satisfies the $M$-condition on $l$.

For any $t < T$, let $c$ be a geodesic realizing the minimum distance between $\partial P_{\epsilon_0}$ and $\partial P_{2\epsilon_0}$. Denote $\epsilon$ to be the geodesic length of $c$. We want to show that $\epsilon$ cannot be too small.

The first case is that $c$ is far away from the vertex. To simply our notations, we can assume the boundary of $P, P_{\epsilon_0}, \partial P_{2\epsilon_0}$ are $x=0, x=\epsilon_0, x=2\epsilon_0$ respectively. We will also suppress $t$ in the following calculations. We let $x, y$ be two axises. Let $C_2$ be a constant such that 
$$
\int_P |Rm|^2 ~ dx < C_2.
$$
Also let
$$
\mathcal{G} = \{ x_0 \in [\epsilon_0, 2\epsilon_0] ~ | ~ \int_{x=x_0} |Rm|^2 < \frac{2C_2}{\epsilon_0} \}.
$$
Then the Lebesgue measure of $\mathcal{G}$ is greater than $\epsilon_0/2$. So $\mathcal{G}$ represents the ``good lines". Let the curve $c$ be parametrized by its Euclidean length: 
$$
c = \{ (x(s), y(s)) ~ | ~ s \in [0, s_0] ~ \}
$$ 
It is easy to see that $s_0 \geq \epsilon_0$. Since 
$$
\int_0^{s_0} \sqrt{(D^2u) (c'(s), c'(s))}  ~ ds= \epsilon, 
$$
we conclude that the Lebesgue measure of the following set
$$
\mathcal{B} = \{ x_0 \in \mathcal{G} ~ | ~ \exists s \in [0, s_0], ~ x(c(s)) = x_0,~ (D^2u)(c'(s), c'(s)) \leq \frac{16 \epsilon^2}{\epsilon_0^2}~  \}
$$
is greater than $\epsilon_0 / 4$. In fact, $\mathcal{B}$ are the bad points in $\mathcal{G} \cap c$. 

On one hand, after choosing $R$ appropriately, Proposition (\ref{bad line}) tells us that
\begin{eqnarray*}
& & \int_{P_{\frac{\epsilon_0}{2}}} Tr (u^{ij}(x)) ~ d\mu \\
&\geq& \int_{x_0 \in \mathcal{B}} \int_{x=x_0} Tr (u^{ij}) ~ d\mu\\
&\geq& \frac{\epsilon_0}{4} \frac{C  \epsilon^2_0}{16\epsilon^2}\\
&=& \frac{C}{\epsilon^2},
\end{eqnarray*}
where $C$ depending only on $\epsilon_0, M, R $ and $\int_P |Rm|^2~ d\mu$.

On the other hand, from Corollary (\ref{interior control}), we know that there exists a constant $C_1 > 0$ such that for any $t < T$, we have
$$
\int_{P_{\epsilon_0}} Tr(u^{ij}(t, x)) ~ d\mu < C_1.
$$
Thus we conclude that $\epsilon$ is bounded from below uniformly.

The remaining case we should consider is that the curve $c$ is close to a vertex. Then without loss of generality, we can let the boundary of $P$ to be $x=0 ~\& ~y=0$ and the boundary of $P_{\epsilon_0}$ to be $x=\epsilon_0 ~\& ~y=\epsilon_0$. Without loss of generality, we can assume that half of the curve $c$, in the sense of Euclidean distance, will touch the set
$$
\{(x,y) \in P_{\epsilon_0} \backslash P_{2\epsilon_0} ~ | ~ x \in [\epsilon, 2\epsilon]~\}.
$$
Then our previous arguments provide the lower bound of $\epsilon$.
\end{proof}

\section{Curvature Control}
In this section, we will use the blow-up analysis to prove Theorem (\ref{key}) which is the key step in the proof of Theorem (\ref{main}).

Suppose that the conclusion is not true, then we can find a sequence of points $(x_i, t_i), x_i \in P_{\epsilon_0}, t_i < T$ such that 
$$
Q(t_i, x_i) d_{u(t_i)}^2(x_i, \partial P_{\epsilon_0}) = \max_{t \leq t_i,~ x \in P_{\epsilon_0}} Q(t, x) d_{u(t)}^2(x, \partial P_{\epsilon_0})
$$
and
$$
\lim_{i \rightarrow \infty} Q(t_i, x_i) d_{g(t_i)}^2(x_i, \partial P_{\epsilon_0}) = \infty.
$$
Let $\lambda_i = Q(t_i, x_i)$, we define a sequence of the Calabi flow $u_i(t), t \leq 0$ as follows:
$$
u_i(t, x) = \lambda_i u\left(\frac{t+t_i}{\lambda_i^2}, \frac{x+x_i}{\lambda_i} \right).
$$ 
Let
$$
Q_i(t,x) =  (|Rm| + |\nabla Rm|^{2/3} + |\nabla^2 Rm|^{1/2})_{u_i} (t,x),
$$
then $Q_i(0,0)=1$. We denote $P^{(i)} = \lambda_i P, P^{(i)}_{\epsilon_0}$ to be $\lambda_i P_{\epsilon_0}$. The following lemma shows that we can pick a backward parabolic neighborhood in the symplectic side for sufficiently large $i$. Note that the Calabi energy and the $M$-condition are scaling invariant.

\begin{prop}
For sufficiently large $i$ and for any $t \in [-1,0]$, we have
$$
d_{u_i(t)} (0, \partial P^{(i)}_{\epsilon_0}) \geq c d_{u_i(0)} (0, \partial P^{(i)}_{\epsilon_0})
$$ 
for some uniform constant $c$.
\end{prop}

\begin{proof}
Without loss of generality, we can just prove for the case $t=-1$. We suppress $i$ in the following calculations. Notice that 
$$
\int_P A^2(-1) ~ d \mu - \int_P A^2(0) ~ d \mu = \int_{-1}^0 \int_P u^{ij}(t) A_{ik}(t) A_{jl}(t) u^{kl}(t) ~ d \mu d t.
$$
Notice that we can assume 
$$
\int_P A^2(-1)  ~ d \mu - \int_P A^2(0) ~ d \mu < 1.
$$
Thus
$$
\int_{-1}^0 \int_P u^{ij}(t) A_{ik}(t) A_{jl}(t) u^{kl}(t) ~ d \mu d t < 1.
$$
For any unit vector $\nu$, we write $u_{11}$ and $A_{11}$ as the second derivative of $u$ and $A$ in the direction of $\nu$. We have
\begin{eqnarray*}
& &|\log u_{11}(0, x) - \log u_{11}(-1,x)| \\
&=& |\int_{-1}^0 \frac{A_{11}}{u_{11}} ~ dt |\\
&\leq& \sqrt{\int_{-1}^0 \frac{A^2_{11}}{u^2_{11}} ~ dt} \\
&\leq& \sqrt{\int_{-1}^0 u^{ij}(t) A_{ik}(t) A_{jl}(t) u^{kl}(t) ~ dt}.
\end{eqnarray*}
In the last inequality, we use the calculation method in Lemma (\ref{coordinate}). We conclude that the Lebesgue measure of the set 
$$
\{ x \in P ~|~\exists \rm{~unit ~vector~} \nu, ~s.t.~ |\log (u_{ij} \nu^i \nu^j)(0, x) - \log (u_{ij} \nu^i \nu^j)(-1,x)| > 1 \}
$$
is less than 1. 

Let $L$ be the geodesic distance between $(0,0)$ and $\partial P_{\epsilon_0}$ at $t=0$. Suppose $\gamma$ is the curve realizing the minimum distance between $(0,0)$ and $\partial P_{\epsilon_0}$ at $t=-1$. Then the geodesic length of $\gamma$ at $t=0$ is greater than $L$. Since $u(t)$ satisfies the $M$-condition, we conclude that the Euclidean distance between the end points of $\gamma$ is greater than $\frac{(\sqrt{2}-1)^2 L^2}{M}$.

The first case is that $\gamma$ is a straight line. We can assume that $\gamma$ lies in the $x$-axis. Notice that $\int_P |Rm|^2 ~ d\mu < C_0.$ Then the Lebesgue measure of the following set
$$
\mathcal{B} = \{ x_0 \in \gamma ~|~ \int_{x=x_0} |Rm|^2 ~ dy > C_0 \}
$$
is less than $1$.

Notice that for any $x \in P$ whose geodesic distance to $(0,0)$ at $t=0$ is less than $L/2$, we have $Q(0,x) \leq 4$. Let $u_{11}$ be the second derivative of $u$ along the $x$-axis. Since $u(t)$ satisfies the $M$-condition, we conclude that $u_{11}(0, x) < C$. Let $(x_0, 0)$ be the point in $\gamma$ such that the geodesic distance between $(0,0)$ and $(x_0,0)$ is $L/2$ at $t=0$. Thus we conclude that
$$
\int_{(x,0) \in \mathcal{B}, ~ x \leq x_0} \sqrt{u_{11}(x)} ~ dx < C.
$$
Since $L  \gg 1$, after removing $\mathcal{B}$ from $\gamma$, the geodesic length of $\gamma$ is still greater than $L/3$ at $t=0$. Our goal is to prove that the geodesic length of $\gamma \backslash \mathcal{B}$ is greater than $c L$ at $t=-1$ for some uniform constant $c$. Thus without loss of generality, we can assume that $\mathcal{B}$ is an empty set.

Now set 
\begin{eqnarray*}
\tilde{\mathcal{B}} &=& \gamma \backslash \{ x \in [0, x_0] ~|~ (D^2 u)(0, (x,0)) \leq 100 (D^2 u)(-1, (x,0)) \} \\
&=& \{ x \in [0, x_0] ~|~ \exists ~ \nu ~s.t. ~ (u^{ij})(0, (x,0))(\nu, \nu) < \frac{1}{100}(u^{ij})(-1, (x,0))(\nu, \nu) \}.
\end{eqnarray*}
We want to show that the Lebesgue measure of $\tilde{\mathcal{B}}$ is well controlled. Once we know this, we can argue that 
$$
\int_{x \in [0, x_0], ~(x,0) \in \tilde{\mathcal{B}}} \sqrt{u_{11}} (0, (x,0)) ~ dx < C.
$$ 
Thus
$$
\int_{x \in [0, x_0], ~(x,0) \notin \tilde{\mathcal{B}}} \sqrt{u_{11}} (0, (x,0)) ~ dx > L/4.
$$ 
Hence
$$
\int_{x \in [0, x_0], ~(x,0) \notin \tilde{\mathcal{B}}} \sqrt{u_{11}} (-1, (x,0)) ~ dx > L/40.
$$
Next we prove that the Lebesgue measure of $\tilde{\mathcal{B}}$ is well controlled. Notice that for every point $x \in \tilde{\mathcal{B}}$, we obtain that there exists a unit vector $\nu$ such that
$$
(u^{ij})_{t=0} (x,0) (\nu, \nu) < \frac{1}{100} (u^{ij})_{t=-1}(x,0) (\nu, \nu).
$$
Applying Lemma (\ref{comparison}), we conclude that there exists a uniform constant $y_0$ such that for each point $y \in [0,y_0]$
$$
(u^{ij})_{t=-1}(x,y) (\nu, \nu) \geq \frac{1}{2} (u^{ij})_{t=-1}(x,0) (\nu, \nu) \geq  50 (u^{ij})_{t=0}(x,0) (\nu, \nu).
$$
Since $u(t)$ satisfies the $M$-condition, the geodesic distance is controlled by the Euclidean distance. Applying Lemma 7 of \cite{D3}, we can choose $y_0$ appropriately such that for any $y \in [0, y_0]$
$$
2 (u^{ij})_{t=0}(x,0) (\nu, \nu) \geq (u^{ij})_{t=0}(x,y) (\nu, \nu).
$$
Thus
$$
 (u^{ij})_{t=-1}(x,y) (\nu, \nu) \geq 25 (u^{ij})_{t=0}(x,y) (\nu, \nu).
$$
Then there is a unit vector $\tilde{\nu}$ such that
$$
\ln (u_{ij})_{t=-1}(x,y) (\tilde{\nu}, \tilde{\nu}) \leq \ln (u_{ij})_{t=0}(x,y) (\tilde{\nu}, \tilde{\nu}) - \ln 25.
$$

It shows that
$$
|\ln (u_{ij})_{t=-1}(x,y) (\tilde{\nu}, \tilde{\nu}) - \ln (u_{ij})_{t=0}(x,y) (\tilde{\nu}, \tilde{\nu})| \geq \ln 25.
$$

Thus the Lebesgue measure of $\tilde{\mathcal{B}} \times [0, y_0]$ is well controlled.

The remaining case is that $\gamma$ is not a straight line.  Let $x$-axis and $y$-axis be perpendicular to each other. The following calculations will be done at $t=0$. Let $p_0$ be the first point in $\gamma$ such that the geodesic distance between $(0,0)$ and $p_0$ is $L/2$. We replace $\gamma$ by the points in $\gamma$ connecting $(0,0)$ and $p_0$. Let $\mathcal{B}_x, \mathcal{B}_y$ be any union of disjoint open intervals in the $x$-axis and $y$-axis respectively. We assume that the Lebesgue measure of $\mathcal{B}_x$ and $\mathcal{B}_y$ are less than $\epsilon$ which is a positive constant to be determined later. Notice that the Riemann length of $\gamma$, i.e., $\bar{L}$, is greater than $\frac{L}{2}$ . We want to show that one of the following is true:
\begin{itemize}
\item For any $\mathcal{B}_x$, after removing the points of $\gamma$ whose $x$-coordinate lies in $\mathcal{B}_x$, the length of $\gamma$ is larger than $c_1 L$.

\item For any $\mathcal{B}_y$, after removing the points of $\gamma$ whose $y$-coordinate lies in $\mathcal{B}_y$, the length of $\gamma$ is larger than $c_1 L$.
\end{itemize}
The above constant $c_1$ is to be determined. Suppose not, then we conclude that the sum of length of intervals in $\gamma$ whose $x$-coordinate in $\mathcal{B}_x$ and $y$-coordinate in $\mathcal{B}_y$ is greater than $\bar{L}- 2 c_1 L$. We want to show that this would lead to a contraction for some $\epsilon$ and $c_1$. Notice that $(D^2 u) < C_0 I_2$ for some uniform constant $C_0$. We have the following observation.

\begin{claim}
Let $\gamma_0 \subset \gamma$ be a curve such that the geodesic distance between the ends point of $\gamma_0$ is $20 \sqrt{C_0}$. Let $|\gamma|_u$ be the geodesic length of $\gamma$. Then there exists $C_2$ depending on $M, \epsilon$ such that
$$
|\gamma_0 \backslash \mathcal{B}_x \times \mathcal{B}_y |_u \geq C_2 20 \sqrt{C_0}.
$$
\end{claim}
Let us assume that our claim holds. Then we can pick a successive sequence of points $p_i \in \gamma, i=1, \ldots, N$ such that for any $i=1, \ldots, N-1$, the geodesic distance between $p_i$ and $p_{i+1}$ is $20 \sqrt{C_0}$. Moreover, the geodesic distance between $p_1$ and $p_N$ is greater than $L/3$. Let $\gamma_{1N}$ be the curve in $\gamma$ connecting $p_1$ and $p_N$. Then our previous claim shows that 
$$
|\gamma_{1N} \backslash \mathcal{B}_x \times \mathcal{B}_y |_u \geq C_2 L/3.
$$
Notice that $|\gamma \backslash \mathcal{B}_x \times \mathcal{B}_y |_u < 2c_1 L$. We derive a contradiction by setting $6c_1 < C_2$. 

Now we turn to give a proof of our previous claim. Without loss of generality, we assume that the two end points of $\gamma_0$ are $(0,0)$ and $p_0$. We can further assume that for any point $p \in \gamma_0$, the geodesic distance between $(0,0)$ and $p$ is less than or equal to $20\sqrt{C_0}$. 

The simple case is that $\langle 1,0 \rangle$ and $\langle 0,1 \rangle$ are the eigenvectors of $(D^2 u)(0,0)$. By Lemma 8 of \cite{D3}, we know that $B_u((0,0),20\sqrt{C_0})$ is almost an ellipsoid. Moreover, for any point $x \in B_u((0,0),20\sqrt{C_0})$, we have a uniform constant $C_3$ such that
$$
\frac{1}{C_3} (D^2 u)(0,0) \leq (D^2 u)(x) \leq C_3 (D^2 u)(0,0).
$$
Let $(x_0, 0)$ be at the boundary of $B_u((0,0),20\sqrt{C_0})$. Notice that $x_0 \geq 20$. Without loss of generality, we can assume that the $x$-coordinate of $p_0$ is greater than $C _4 x_0$ for some uniform constant $C_4$. We have the following inequalities:
\begin{eqnarray*}
|\gamma_0 |_u \geq \frac{1}{\sqrt{C_3}} \sqrt{u_{11}}(0,0) (C_4 x_0 - \epsilon) \geq C_2 20 \sqrt{C_0}.
\end{eqnarray*}
The more complicated case is that we need to rotate the coordinate systems to get $x'$-axis and $y'$-axis so that the $\langle 1,0 \rangle$ and $\langle 0,1 \rangle$ axis are eigenvector of $(D^2 u)(0,0)$. Notice that we can always find new $\mathcal{B}'_{x'}$ and $\mathcal{B}'_{y'}$ with Lebesgue measure less than $2 \epsilon$ such that 
$$
\mathcal{B}_x \times \mathcal{B}_y \subset \mathcal{B}'_{x'} \times \mathcal{B}'_{y'}.
$$
Then our earlier arguments apply. So without loss of generality, we have that for any $\mathcal{B}_x$ with Lebesgue measure less than $\epsilon$, after removing the set of points in $\gamma$ whose $x$-coordinate is in $\mathcal{B}_x$, the length of $\gamma$ is still greater than $c_1 L$. Then we can apply the arguments in the case where $\gamma$ is a straight line to obtain the conclusion as follows: 

\begin{itemize}
\item By choosing $C_5 > 0$ appropriately, the Lebesgue measure of the following set in $x$-axis is less than $\epsilon/2$ :
$$
\mathcal{B}_1 = \{ x_0~ | ~\exists~ y_0 ~ s.t.~ (x_0, y_0) \in \gamma~ \&~ \int_{x=x_0} |Rm|^2(x_0, y) ~ dy \geq C_5 \}.
$$

\item Let $\gamma_0$ be the sets of points in $\gamma$ whose geodesic distance to $(0,0)$ at $t=0$ is less than $L/2$. By choosing $C_6 > 0$ appropriately, the Lebesgue measure of the following set in $x$-axis is less than $\epsilon/2$ :
\begin{eqnarray*}
\mathcal{B}_2 =& \{ & x_0 \notin \mathcal{B}_1~ | ~\exists~ y_0 ~ s.t.~ (x_0, y_0) \in \gamma_0~  \mbox{and}~\\
& & \exists~ \nu~ s.t.~ (u^{ij})_{t=0}(x_0, y_0)(\nu, \nu) < C_6 (u^{ij})_{t=-1}(x_0, y_0)(\nu, \nu)  \}.
\end{eqnarray*}
\end{itemize}

Then after removing $(\mathcal{B}_1 \cup \mathcal{B}_2) \times y$ from $\gamma$, we know that the length of $\gamma$ at $t=0$ is greater than $c_1 L$ for some uniform constant $c_1 > 0$, then the length of $\gamma$ at $t=-1$ is also greater than $c L$ for some uniform constant $c$.
\end{proof}

\subsection{Evolution equation of the Calabi flow with a cut-off function}
By the previous discussions, we can conclude that 
\begin{itemize}

\item For any point $(t, p) \in [-1,0] \times P_{\epsilon_0}$,  
$$Q(t,p) \times d^2_t (p, \partial  P_{\epsilon_0}) \leq L.$$

\item For any $t \in [-1, 0], d_t ((0,0),  \partial  P_{\epsilon_0}) \geq c \sqrt{L}.$

\item $d_0 ((0,0),  \partial  P_{\epsilon_0}) = \sqrt{L}.$

\end{itemize}
Thus without loss of generality, we can assume that $|Rm|(t, x) \leq 4$ in $B_{u_t}((0,0), 2)$ and $|Rm|(0,0) = 1$ at $t=0$.

Now let us consider a $C^2$ cutoff function $\psi  : [0,2] \rightarrow [0,1]$ such that
\begin{itemize}
\item $$\psi(t)=1,\  0 \leq t \leq 1/2$$
\item $$\psi(t)=0,\  t \geq 1$$
\item $$|\psi'(t)| \leq C \psi(t)^{\frac{a-1}{a}}$$
\item $$|\psi''(t)| \leq C \psi(t)^{\frac{a-1}{a}}$$ where $ a \in \mathbb{Z}^+$ to be determined.
\end{itemize}

Let $r_t (\cdot) = d_t (\cdot, (0,0))$ be the Riemannian distance function in the totally geodesic submanifold $P$ with induced metric $$g^{(t)}=u^{(t)}_{ij} d x_i d x_j.$$ It is naturally to extend $r$ to be a function in the whole manifold such that it is invariant under the $\mathbb{T}^2$ action. In this case, $r$ is the distance function to a submanifold $\mathbb{T}^2$ and we define $f_t = \psi(r_t)$ to be a cutoff function. We suppress $t$ in the following calculations.

It is easy to see that away from \{the cut locus of $(0,0) \subset P$ \} $\times \mathbb{T}^2$, $$|\nabla f| = |\psi'| |\nabla r| \leq C \psi(t)^{\frac{a-1}{a}}$$ and $$\triangle f = \psi'' |\nabla r|^2 + \psi' \triangle r.$$ Since the curvature is bounded, we can show that $|\triangle r|$ is also bounded for $0 \leq r \leq 1$ by comparison geometry.

First we are trying to argue that $\triangle r$ is bounded by an universal constant $C$ away from \{the cut locus of $(0,0)$ \} $\times \mathbb{T}^2$. Notice that we only need to pay attention to the case where $1/2 \leq r \leq 1$.

To get the upper bound of $\triangle r$, we apply the Weitzenb\"ock formula, i.e.

$$
|\mbox{Hess} \ r|^2 + \frac{\partial}{\partial r}(\triangle r) + \mbox{Ric}(\frac{\partial}{\partial r}, \frac{\partial}{\partial r}) = 0
$$

Let $\lambda_1, \ldots, \lambda_{4}$ be the eigenvalues of Hess $r$. Then
$$
|\mbox{Hess} \ r|^2 = \lambda_1^2 + \cdots + \lambda_n^2 \geq \frac{(\lambda_1 + \cdots + \lambda_4)^2}{4} = \frac{(\mbox{tr}(\mbox{Hess}\ r))^2}{4}=\frac{(\triangle r)^2}{4}
$$

Thus,

$$
\frac{(\triangle r)^2}{4} + \frac{\partial}{\partial r} (\triangle r) + C \leq 0
$$

Let $\phi=\frac{4}{\triangle r}$. Then
\begin{eqnarray}
\label{WF}
\frac{\phi'}{1 + C \phi^2} \geq 1,
\end{eqnarray}
where $C < 0$ is an universal constant.

Suppose we are evaluating $\triangle r$ at the point $x$ and $\gamma$ is the unique minimizing geodesic connecting $p=(0,0), x \in P$. It is also easy to see that $\gamma$ is the minimizing geodesic in the whole manifold. Taking $x_1, x_2, \ldots$ approaching $p$ along $\gamma$, and rescaling the metric by $\frac{1}{r(x_i)}$, we getting a sequence of manifold converging to standard $\mathbb{R}^4$ in Cheeger-Gromov sense. Using the geodesic spherical coordinate in $P$, we can write the metric in the whole manifold (locally) as
$$
g = r \otimes r + h_{ij} d \theta_i \otimes d \theta_j + g^{ij} d \eta_i \otimes d \eta_j
$$

Hence in the limiting process, $r$ is always fixed in the coordinate system. Since the metric converges in $C^{\infty}$, we conclude that $r(x_i) \triangle r(x_i)$ converges to 1. Thus $\phi(x_i) \sim 4 r(x_i)$, plugging into inequality (\ref{WF}) and integrating it gives

$$
\triangle r \leq 4 \sqrt{-C} \coth( \sqrt{-C} r)
$$

For the lower bound of $\triangle r$, if $\triangle r(x)$ is very negative, then $\phi$ is very close to 0, hence $1 + C \phi^2 > C_1 > 0$. Let $\gamma(s)$ be a geodesic connecting $\gamma(0) = p$ and $\gamma(s_0) = x$ where $s$ is the arc-length parameter of $\gamma$. Then $\phi(s_0+s)$ approaches to 0 as $s>0$ increases. Inequality (\ref{WF}) tells us $\phi$ will reach to 0 for a small $s > 0$. Thus we obtain a contradiction if $s_0$ is not close to a conjugate point in $\gamma$.

Now we have the following lemma:

\begin{lemma} If $|Rm| \leq 4$ in $B_u (p, 2)$, then we can construct a cutoff function $f$ such that
$$|\nabla f| \leq C f^{\frac{a-1}{a}}, \quad |\triangle f| \leq C f^{\frac{a-1}{a}},$$
where $C$ is a universal constant.
\end{lemma}

Next, we are going to derive a sequence of integral inequalities which will be useful in calculating the evolution equation of the Calabi flow with a cut-off function. We will do the calculations in the complex side. Please keep in mind that our cut-off function only define in the polygon and we do not have any control in the torus direction. For example, the integral of $f$ is not necessary bounded. In Proposition (\ref{e-inequality3}), we bypass this difficulty by using that the integral of $f |Rm|^2$ is bounded. To simplify the notations, we write $\int_X f ~ dg$ as $\int f$.

\begin{lemma}
\label{e-inequality2}

For every $k \in \mathbb{Z}^+$, there is $a(k) \in \mathbb{Z}^+,  0 < b(k) < 1 $ such that for all $a > a(k), 1 >  b > b(k)$,

$$
\int f^b |\nabla^k Rm|^2 \leq \epsilon \int f |\nabla^{k+1} Rm|^2 + C(k,\epsilon)
$$
\end{lemma}

\begin{proof}
We derive the inequality by induction. For $k=0$, it is obvious. For $k>0$, we have
\begin{eqnarray*}
& & \int f^{b} |\nabla^k Rm|^2 \\
&=& \int \nabla f^{b} * \nabla^{k-1} Rm * \nabla^k Rm + \int f^{b} \nabla^{k+1} Rm * \nabla^{k-1} Rm \\
&\leq& \epsilon \int f^{b} |\nabla^k Rm|^2 + C(\epsilon) \int f^{b-\frac{2}{a}} |\nabla^{k-1} Rm|^2 \\
& &+ \epsilon \int f |\nabla^{k+1} Rm|^2 + C(\epsilon) \int f^{2b-1} |\nabla^{k-1} Rm|^2.
\end{eqnarray*}

Choose $b(k), a(k)$ such that $$2b(k)-1 > b(k-1), \quad b(k)-\frac{2}{a(k)} > b(k-1),\quad a(k) > a(k-1),$$ then by induction, we have
$$
C(\epsilon) \int f^{b-\frac{2}{a}} |\nabla^{k-1} Rm|^2 \leq \epsilon \int |\nabla^k Rm|^2 + C(k, \epsilon)
$$
and
$$
C(\epsilon) \int f^{2b-1} |\nabla^{k-1} Rm|^2 \leq \epsilon \int |\nabla^k Rm|^2 + C(k, \epsilon).
$$
Thus we obtain the desired inequality.
\end{proof}

\begin{cor}
\label{e-inequality}
For every $k \in \mathbb{Z}^+$ and $\epsilon > 0$, there are $\bar{a}(k), C(k,\epsilon)$ such that for all $a \geq \bar{a}(k)$
\begin{eqnarray}
\int f |\nabla^i Rm|^2 < \epsilon \int f |\nabla^k Rm|^2 + C(k,\epsilon),  \quad \mbox{for} \quad 0 < i < k.
\end{eqnarray}
\end{cor}

\begin{proof}
We derive the inequality by induction. For $k=1$, it is obvious. For $k > 1$, we have

\begin{eqnarray*}
& & \int f |\nabla^{k-1} Rm|^2 \\
&=& \int f \nabla^k Rm * \nabla^{k-2} Rm + \int \nabla f * \nabla^{k-1} Rm * \nabla^{k-2} Rm\\
&\leq& \epsilon \int f |\nabla^k Rm|^2 + C(\epsilon) \int f |\nabla^{k-2} Rm|^2 \\
& &+ \epsilon \int f |\nabla^{k-1} Rm|^2 + C(\epsilon) \int f^{\frac{a-2}{a}} |\nabla^{k-2} Rm|^2.
\end{eqnarray*}

Hence

$$
\int f |\nabla^{k-1} Rm|^2 \leq \epsilon \int f |\nabla^k Rm|^2 + C(k, \epsilon) + C(\epsilon) \int f^{\frac{a-2}{a}} |\nabla^{k-2} Rm|^2.
$$

We choose $\bar{a}(k)$ such that 
$$\frac{\bar{a}(k)-2}{\bar{a}(k)} > b(k-2), \quad \bar{a}(k) \geq a(k).$$ Applying Proposition (\ref{e-inequality2}), we obtain the result.
\end{proof}

\begin{prop}
\label{e-inequality3}
For every positive integer $k \geq 2$ and every $p \geq 1$, there exists constants $c(k,p) \in \mathbb{Z}^+, C(k,p) > 0$ depending only on $k, p$ such that for $a > c(k,p)$ and $0 < i < k, $
$$
\int f |\nabla^i Rm|^{\frac{2p k}{i}} \leq C(k,p) \left(1 + \int f |\nabla^k Rm|^{2p} + \int f |\nabla^k Rm|^2 \right).
$$
\end{prop}

\begin{proof}
We will derive this inequality by induction on $k$. Notice that $|Rm|$ is bounded in the support of $f$. For $k=2$, we have
\begin{eqnarray*}
& & \int f |\nabla Rm|^{4p} \\
&\leq& \int |\nabla f| |Rm| |\nabla Rm|^{4p-1} + C \int f |\nabla^2 Rm| |Rm| |\nabla Rm|^{4p-2} \\
&\leq& \int \left( \frac{1}{4} f |\nabla Rm|^{4p} + C f^{\frac{a-4p}{a}} |Rm|^{4p} \right) + \int \left( \frac{1}{4} f |\nabla Rm|^{4p} + C f |\nabla^2 Rm|^{2p} \right).
\end{eqnarray*}

Since $f^{\frac{a-4p}{a}} |Rm|^{4p} \leq C |Rm|^2$, we have

$$
\int f |\nabla Rm|^{4p} \leq C ( 1 + \int f |\nabla^2 Rm|^{2p} ).
$$

Now let us assume that the inequality holds up to $k-1$. Let $c(k,p) > \max({\bar{a}(k), c(k-1, pk/(k-1))})$. Then for any $i < k-1$, by induction and Corollary (\ref{e-inequality}), we have
$$
\int f |\nabla^i Rm|^{\frac{2p k}{i}} \leq C \left(1 + \int f |\nabla^{k-1} Rm|^{\frac{2pk}{k-1}} + \int f |\nabla^k Rm|^2 \right).
$$
Thus we only need to show that
$$
\int f |\nabla^{k-1} Rm|^{\frac{2p k}{k-1}} \leq C \left(1 + \int f |\nabla^k Rm|^{2p} + \int f |\nabla^k Rm|^2 \right).
$$

Using integration by parts, we get

\begin{eqnarray*}
& & \int f |\nabla^{k-1} Rm|^{\frac{2p k}{k-1}}  \\
&\leq& \int |\nabla f| |\nabla^{k-2} Rm| |\nabla^{k-1} Rm|^{\frac{2pk}{k-1}-1} + C \int f |\nabla^k Rm| |\nabla^{k-2} Rm| |\nabla^{k-1} Rm|^{\frac{2pk}{k-1}-2} \\
&\leq& \int \left( \frac{1}{4} f |\nabla^{k-1} Rm|^{\frac{2pk}{k-1}} + C f^{1-\frac{2pk}{a(k-1)}} |\nabla^{k-2} Rm|^{\frac{2pk}{k-1}} \right)+ \\
& & \int f \left( C |\nabla^k Rm|^{2p} + \frac{1}{4}|\nabla^{k-1} Rm|^{\frac{2pk}{k-1}} + \epsilon(p,k) |\nabla^{k-2} Rm|^{\frac{2pk}{k-2}} \right),
\end{eqnarray*}
where $\epsilon(p,k)$ is to be determined. The only term we need to worry about is

$$
\int f^{1-\frac{2pk}{a(k-1)}} |\nabla^{k-2} Rm|^{\frac{2pk}{k-1}}.
$$

Now let
$$
\frac{1}{s}=\frac{\frac{pk}{k-1}-1}{\frac{pk}{k-2}-1}, \quad \frac{1}{r} = 1 - \frac{1}{s}.
$$

Using Young's inequality

$$ xy \leq \frac{x^r}{r} + \frac{y^s}{s},$$

we obtain

$$
f^{1-\frac{2pk}{a(k-1)}} |\nabla^{k-2} Rm|^{\frac{2pk}{k-1}} \leq f^{1-r \frac{2pk}{a(k-1)}} \frac{|\nabla^{k-2} Rm|^2}{r} + f \frac{|\nabla^{k-2} Rm|^{\frac{2pk}{k-2}}}{s}.
$$

It is easy to see that by further increasing $c(p,k)$ then
$$ b = 1-r \frac{2pk}{a (k-1)} $$ would satisfies the assumption of Lemma (\ref{e-inequality2}). Hence

\begin{eqnarray*}
& & \int f^{1-\frac{2pk}{a(k-1)}} |\nabla^{k-2} Rm|^{\frac{2pk}{k-1}} \\
&\leq& C \int f^{1-r \frac{2pk}{a(k-1)}} |\nabla^{k-2} Rm|^2 + \epsilon(p,k) \int f |\nabla^{k-2} Rm|^{\frac{2pk}{k-2}} \\
&\leq& C \left( 1 + \int f |\nabla^{k-1} Rm|^2 \right) +C \epsilon(p,k) \int f |\nabla^{k-1} Rm|^{\frac{2pk}{k-1}}.
\end{eqnarray*}

We obtain the conclusion by choosing $\epsilon(p,k)$ sufficiently small.
\end{proof}

\begin{thm}
\label{evolution inequality}

Under our settings, we get the following evolution inequality

$$
\frac{\partial}{\partial t}\int_X f |\nabla^k Rm|^2 dg \leq -\frac{1}{4} \int_X f |\nabla^{k+2} Rm|^2 dg + C
$$

in the sense of distribution, where $C$ is a constant depending only on $k$ and the Calabi energy.
\end{thm}

Before we get into the proof, let us derive the evolution formula for
$$
\left( \frac{\partial}{\partial t} + \triangle^2 \right) |\nabla^k Rm|^2.
$$

Let us recall that
$$
\frac{\partial Rm}{\partial t} = - \triangle^2 Rm + \nabla Rm * \nabla Rm + \nabla^2 Rm * Rm,
$$
where $\nabla = \partial + \bar{\partial}$ and $\triangle = \partial \bar{\partial} = \bar{\partial} \partial$. Then
$$
\frac {\partial \nabla^k Rm}{\partial t} = -\triangle^2 \nabla^k Rm + \sum_{i+j=k+2} \nabla^i Rm * \nabla^j Rm.
$$
So,
\begin{eqnarray*}
& & \frac{\partial}{\partial t} |\nabla^k Rm|^2 \\
&=& \langle \frac{\partial}{\partial t} \nabla^k Rm, \overline{\nabla^k Rm} \rangle + \langle \nabla^k Rm, \frac{\partial}{\partial t} \overline{\nabla^k Rm} \rangle + \nabla^2 Rm * \nabla^k Rm * \nabla^k Rm\\
&=& - \langle \triangle^2 \nabla^k Rm, \overline{\nabla^k Rm} \rangle - \langle \nabla^k Rm, \triangle^2 \overline{\nabla^k Rm} \rangle \\
& & + \sum_{i+j=k+2} \nabla^i Rm * \nabla^j Rm * \nabla^k Rm.
\end{eqnarray*}

Moreover,

\begin{eqnarray*}
& & \triangle^2 |\nabla^k Rm|^2 \\
&=& \triangle \langle \triangle \nabla^k Rm, \overline{\nabla^k Rm} \rangle + \triangle \langle \nabla^k Rm, \triangle \overline{\nabla^k Rm} \rangle \\
& & + \triangle \left( \langle \nabla_i \nabla^k Rm, \nabla_{\bar{i}} \overline{\nabla^k Rm} \rangle + \langle \nabla_{\bar{i}} \nabla^k Rm, \nabla_i \overline{\nabla^k Rm} \rangle \right) \\
&=& \langle \triangle^2 \nabla^k Rm , \overline{\nabla^k Rm} \rangle + \langle \nabla_i \triangle \nabla^k Rm, \nabla_{\bar{i}} \overline{\nabla^k Rm} \rangle\\
& & + \langle \nabla_{\bar{i}} \triangle \nabla^k Rm, \nabla_i \overline{\nabla^k Rm} \rangle + \langle \triangle \nabla^k Rm , \triangle \overline{\nabla^k Rm} \rangle \\
& & + \langle \triangle \nabla^k Rm , \triangle \overline{\nabla^k Rm} \rangle + \langle \nabla_i \nabla^k Rm, \nabla_{\bar{i}} \triangle \overline{\nabla^k Rm} \rangle\\
& & + \langle \nabla_{\bar{i}} \nabla^k Rm, \nabla_i  \triangle \overline{\nabla^k Rm} \rangle + \langle \nabla^k Rm , \triangle^2 \overline{\nabla^k Rm} \rangle \\
& & + \triangle \left( |\nabla^{k+1} Rm|^2 \right). \\
\end{eqnarray*}

Since
\begin{eqnarray*}
& & \langle \nabla_i \triangle \nabla^k Rm, \nabla_{\bar{i}} \overline{\nabla^k Rm} \rangle + \langle \nabla_i \nabla^k Rm, \nabla_{\bar{i}} \triangle \overline{\nabla^k Rm} \rangle \\
&=& \triangle |(\nabla^{k} Rm)_{, i}|^2 - |(\nabla^k Rm)_{, i j}|^2 - |(\nabla^k Rm)_{, i \bar{j}}|^2 \\
& & + Rm * \nabla^{k+1} Rm * \nabla^{k+1} Rm + \nabla Rm * \nabla^k Rm * \nabla^{k+1} Rm
\end{eqnarray*}

and

\begin{eqnarray*}
& & \langle \nabla_i \nabla^k Rm, \nabla_{\bar{i}} \triangle \overline{\nabla^k Rm} \rangle + \langle \nabla_{\bar{i}} \nabla^k Rm, \nabla_i  \triangle \overline{\nabla^k Rm} \rangle\\
&=& \triangle |(\nabla^{k} Rm)_{, \bar{i}}|^2 - |(\nabla^k Rm)_{, \bar{i} j}|^2 - |(\nabla^k Rm)_{, \bar{i} \bar{j}}|^2 \\
& & + Rm * \nabla^{k+1} Rm * \nabla^{k+1} Rm + \nabla Rm * \nabla^k Rm * \nabla^{k+1} Rm.
\end{eqnarray*}

We get

\begin{eqnarray*}
& & \left( \frac{\partial}{\partial t} + \triangle^2 \right) |\nabla^k Rm|^2 \\
&=& - |(\nabla^k Rm)_{, i j}|^2 - |(\nabla^k Rm)_{, i \bar{j}}|^2 - |(\nabla^k Rm)_{,\bar{i} j}|^2 - |(\nabla^k Rm)_{,\bar{i} \bar{j}}|^2 + 2 |\triangle \nabla^k Rm|^2 \\
& & + 2 \triangle \left( |\nabla^{k+1} Rm|^2 \right) + \sum_{i+j=k+2} \nabla^i Rm * \nabla^j Rm * \nabla^k Rm + Rm * \nabla^{k+1} Rm * \nabla^{k+1} Rm.
\end{eqnarray*}

Now we are ready to prove Theorem (\ref{evolution inequality})

\begin{proof}

\begin{eqnarray*}
& & \frac{\partial}{\partial t} \int_X  f |\nabla^k Rm|^2 \ dg\\
&=& \int_X \frac{\partial}{\partial t} \left(  f |\nabla^k Rm|^2 \right) + f |\nabla^k Rm|^2 \triangle R \ dg\\
&=& \int_X \left( \frac{\partial}{\partial t} + \triangle^2 \right)  \left( f |\nabla^k Rm|^2 \right) + f |\nabla^k Rm|^2 \triangle R \ dg\\
&=& \int_X f \left( \frac{\partial}{\partial t} + \triangle^2 \right) \left( |\nabla^k Rm|^2 \right) + \\
& & \frac{\partial f}{\partial t} |\nabla^k Rm|^2 + f \nabla^k Rm * \nabla^k Rm * \nabla^2 Rm - (\triangle f) (\triangle |\nabla^k Rm|^2) \ dg \\
&=& \int_X f \left( - |(\nabla^k Rm)_{, i j}|^2 - |(\nabla^k Rm)_{, i \bar{j}}|^2 - |(\nabla^k Rm)_{,\bar{i} j}|^2 - |(\nabla^k Rm)_{,\bar{i} \bar{j}}|^2 + 2 |\triangle \nabla^k Rm|^2 \right) \ dg + \\
& & \int_X f \left( \sum_{i+j=k+2} \nabla^i Rm * \nabla^j Rm * \nabla^k Rm + Rm * \nabla^{k+1} Rm * \nabla^{k+1} Rm \right) \ dg+ \\
& & \int_X \frac{\partial f}{\partial t} |\nabla^k Rm|^2 + (\triangle f) (  2 |\nabla^{k+1} Rm|^2  - \triangle |\nabla^k Rm|^2) \ dg
\end{eqnarray*}

Now let us deal with the above three integral one by one. For the first intergral, since

\begin{eqnarray*}
& & - \int_X f |(\nabla^k Rm)_{, i j}|^2 \\
&=& \int_X f_{\bar{i}} (\nabla^k Rm)_{, i j} (\nabla^k Rm)_{, \bar{j}} + f (\nabla^k Rm)_{, i j \bar{i}} (\nabla^k Rm)_{, \bar{j}}\\
&=& \int_X \nabla f * \nabla^{k+2} Rm * \nabla^{k+1} Rm + f (\nabla^k Rm)_{, i \bar{i} j} (\nabla^k Rm)_{, \bar{j}} + f Rm * \nabla^{k+1} Rm * \nabla^{k+1} Rm \\
&=& \int_X \nabla f * \nabla^{k+2} Rm * \nabla^{k+1} Rm - f (\nabla^k Rm)_{, i \bar{i}} (\nabla^k Rm)_{, \bar{j} j} + f Rm * \nabla^{k+1} Rm * \nabla^{k+1} Rm\\
&=& \int_X \nabla f * \nabla^{k+2} Rm * \nabla^{k+1} Rm + f \nabla^{k+2} Rm * \nabla^{k} Rm * Rm \\
& & - f (\nabla^k Rm)_{, i \bar{i}} (\nabla^k Rm)_{, j \bar{j}} + f Rm * \nabla^{k+1} Rm * \nabla^{k+1} Rm\\
&\leq& \int_X -|\triangle \nabla^k Rm|^2 + \epsilon f |\nabla^{k+2} Rm|^2 + C f^{\frac{a-2}{a}} |\nabla^{k+1} Rm|^2 + C f |\nabla^{k} Rm|^2 + C f |\nabla^{k+1} Rm|^2
\end{eqnarray*}

By Lemma (\ref{e-inequality2}) and Corollary (\ref{e-inequality}), we get
$$
- \int_X f |(\nabla^k Rm)_{, i j}|^2 \leq  C + \int_X -|\triangle \nabla^k Rm|^2 + \epsilon \int_X f |\nabla^{k+2} Rm|^2
$$

Also
\begin{eqnarray*}
& & - \int_X f |(\nabla^k Rm)_{, i \bar{j}}|^2 \\
&=& - \int_X f (\nabla^k Rm)_{, i \bar{j}} (\nabla^k Rm)_{, j \bar{i}} + f \nabla^{k+2} Rm * \nabla^k Rm * Rm \\
&=& \int_X f_{\bar{j}} (\nabla^k Rm)_{, i} (\nabla^k Rm)_{, \bar{i} j} + f (\nabla^k Rm)_{, i } (\nabla^k Rm)_{, j \bar{i} \bar{j}} + f \nabla^{k+2} Rm * \nabla^k Rm * Rm\\
&=& \int_X \nabla f * \nabla^{k+2} Rm * \nabla^{k+1} Rm + f (\nabla^k Rm)_{, i} (\nabla^k Rm)_{, j \bar{j} \bar{i} } + f \nabla^{k+2} Rm * \nabla^k Rm * Rm \\
&=& \int_X \nabla f * \nabla^{k+2} Rm * \nabla^{k+1} Rm - f (\nabla^k Rm)_{, i \bar{i}} (\nabla^k Rm)_{,j \bar{j}} + f \nabla^{k+2} Rm * \nabla^k Rm * Rm\\
&\leq& \int_X -|\triangle \nabla^k Rm|^2 + \epsilon f |\nabla^{k+2} Rm|^2 + C f^{\frac{a-2}{a}} |\nabla^{k+1} Rm|^2 + C f |\nabla^{k} Rm|^2.
\end{eqnarray*}

By Lemma (\ref{e-inequality2}) and Corollary (\ref{e-inequality}), we get
$$
- \int_X f |(\nabla^k Rm)_{, i \bar{j}}|^2 \leq  C + \int_X -|\triangle \nabla^k Rm|^2 + \epsilon \int_X f |\nabla^{k+2} Rm|^2
$$

Hence
\begin{eqnarray*}
& & \int_X f \left( - |(\nabla^k Rm)_{, i j}|^2 - |(\nabla^k Rm)_{, i \bar{j}}|^2 - |(\nabla^k Rm)_{,\bar{i} j}|^2 - |(\nabla^k Rm)_{,\bar{i} \bar{j}}|^2 + 2 |\triangle \nabla^k Rm|^2 \right) \\
&\leq& C - (\frac{1}{2} -\epsilon) \int_X f |\nabla^{k+2} Rm|^2
\end{eqnarray*}

Now we estimate the second integral. For $i > 0, j >0$, by Proposition (\ref{e-inequality3}), we have

\begin{eqnarray*}
& & \left| \int_X f \nabla^i Rm * \nabla^j Rm * \nabla^k Rm \right| \\
&\leq& \left( \int_X f |\nabla^i Rm|^{\frac{2k+4}{i}} \right)^{\frac{i}{2k+4}} \left( \int_X f |\nabla^j Rm|^{\frac{2k+4}{j}} \right)^{\frac{j}{2k+4}} \left( \int_X f |\nabla^k Rm|^2 \right) ^{\frac{1}{2}} \\
&\leq& \left( C + C \int_X f |\nabla^{k+2} Rm|^2 \right)^{\frac{i}{2k+4}} \left( C + C \int_X f |\nabla^{k+2} Rm|^2 \right)^{\frac{j}{2k+4}} \left( C +  \epsilon \int_X f |\nabla^{k+2} Rm|^2 \right) ^{\frac{1}{2}} \\
&\leq& C + \epsilon \int_X f |\nabla^{k+2} Rm|^2.
\end{eqnarray*}

Also we have
$$
\left|\int_X f Rm * \nabla^{k} Rm * \nabla^{k+2} Rm \right| \leq  C + \epsilon \int_X f |\nabla^{k+2} Rm|^2,
$$
$$
\left|\int_X f Rm * \nabla^{k+1} Rm * \nabla^{k+1} Rm \right| \leq  C + \epsilon \int_X f |\nabla^{k+2} Rm|^2.
$$

Now we try to control the third integral. Notice that we pick $(0,0)$ in the symplectic side. So it will move in the complex side as $t$ changes. By Legendre transformation, the coordinate of $(0,0)$ in the complex side is $\xi_t = D u_t(0,0)$. We set this curve to be $\tilde{c}(s)$, parameterizing by $s=t$. For any point $\xi$ in the complex side, let $\gamma_t$ be the minimizing geodesic connecting $\xi_t$ and $\xi$ at time $t$. We first control the upper bound of $\frac{\partial d_t}{\partial t}(\xi, \xi_t)$ at $t=t_0$. Let $t \rightarrow t_0^+$. We can replace $\gamma_t$ by $\gamma_{t_0}$ plus the part of $\tilde{c}$ connecting $\xi_{t_0}$ and $\xi_t$. Let $\bar{\gamma}_t$ be the part of $\tilde{c}$ connecting $\xi_{t_0}$ and $\xi_t$. Let $\tilde{d}_t$ be its Riemann length at $t$. Then
$$
\frac{\partial \tilde{d}_t}{\partial t} (t_0) = |\nabla A_{t_0}|_{u_{t_0}} (0,0).
$$
By our definition of $Q$, we know that $|\nabla A_{t_0}|_{u_{t_0}}(0,0)$ is bounded.  Let $|\gamma|_u$ be the Riemannian length of $\gamma$. Then in the sense of distribution, we have
$$
\frac{\partial d_t(\xi, \xi_t)}{\partial t} |_{t=t_0} \leq  \frac{\partial |\gamma_{t_0}|_{u_t}}{\partial t}(t_0) + C.
$$
Notice that if the distance between $\xi_0$ and $\xi$ is less than 2, we have
$$
\frac{\partial |\gamma_{t_0}|_{u_t}}{\partial t}(t_0)  \leq \int_{\gamma_{t_0}} |\nabla^2 A|_u < C.
$$
So we have obtained the upper bound of $\frac{\partial d_t}{\partial t}(\xi, \xi_t)$ at $t=t_0$ in the sense of distribution. Next we want to get the lower bound of $\frac{\partial d_t}{\partial t}(\xi, \xi_t)$ at $t=t_0$ in the sense of distribution. Notice that $$d_t (\xi,\xi_t) \geq d_t (\xi, \xi_0) - |\bar{\gamma}_t|_{u_t}.$$ So we only need to get a lower bound of $\frac{\partial d_t(\xi, \xi_0)}{\partial t}|$ at $t=t_0$. It is well known that
$$
\frac{\partial d_t(\xi, \xi_0)}{\partial t}(\xi, \xi_t) \geq \frac{\partial |\gamma_{t_0}|_{u_t}}{\partial t}(t_0)  \geq - \int_{\gamma_{t_0}} |\nabla^2 A|_u > - C.
$$

So we conclude that $\frac{\partial f}{\partial t}$ is bounded in the sense of distribution.

Together with

$$
\triangle |\nabla^k Rm|^2 = \nabla^{k+2} Rm * \nabla^k Rm + \nabla^{k+1} Rm * \nabla^{k+1} Rm,
$$

we have
\begin{eqnarray*}
& & \left| \int_X \frac{\partial f}{\partial t} |\nabla^k Rm|^2 + (\triangle f) ( 2 |\nabla^{k+1} Rm|^2 - \triangle |\nabla^k Rm|^2) \right| \\
&\leq& \int_X C f^{\frac{a-1}{a}} |\nabla^k Rm|^2 + \epsilon f |\nabla^{k+2} Rm|^2 + C f^{\frac{a-2}{a}} |\nabla^k Rm|^2 + C f^{\frac{a-1}{a}} |\nabla^{k+1} Rm|^2\\
&\leq& C + \epsilon \int_X f |\nabla^{k+2} Rm|^2.
\end{eqnarray*}

Combining the above results together, we get the conclusion.
\end{proof}

Now let us define
$$
F_k(t)=\sum_{i=0}^k t^i \int_X f |\nabla^i Rm|^2 \ d g(t-1),
$$
where $t \in [0,1]$. Then
\begin{eqnarray*}
\frac{\partial F_k}{\partial t} &=& \sum_{i=1}^k it^{i-1} \int_X f |\nabla^i Rm|^2 + \sum_{i=0}^k t^i \frac{\partial}{\partial t} \int_X f |\nabla^i Rm|^2 \\
&\leq& \sum_{i=1}^k it^{i-1} \int_X f |\nabla^i Rm|^2 + \sum_{i=0}^k t^i \left( -\frac{1}{4} \int_X f |\nabla^{i+2} Rm|^2 + C \right) \\
&\leq& \sum_{i=0}^{k-1} t^i  \left( -\frac{1}{4} \int_X f |\nabla^{i+2} Rm|^2 + (i+1)\int_X f |\nabla^{i+1} Rm|^2 \right) + C\\
&\leq& C
\end{eqnarray*}

Hence $F(1)-F(0) \leq C$. Then we have the following result:

\begin{cor}
At time $t=0$, for every $k \in \mathbb{Z}^+$, there is a constant $C(k)$ depending only on $k$, such that

$$
\int_X f |\nabla^k Rm|^2 \ dg < C(k)
$$

Especially,

$$
\int_{B(p_0,\frac{1}{2})} |\nabla^k Rm|^2 \ d \mu< C(k)
$$

where $p_0 = (0,0), B(p_0,\frac{1}{2})$ is a geodesic ball in $P$ and $d \mu$ is the standard Euclidean measure on $P$. 
\end{cor}
Now applying the arguments in the appendix of \cite{FH}, we have
\begin{cor}
For any $p \in B(p_0,\frac{1}{2})$.
$$
|\nabla^k Rm| (p)< C(k).
$$
\end{cor}

Moreover we have
\begin{cor}
For any $C > 0$, there exists an $i(C) \in \mathbb{Z}^+$ such that for any $i > i(C), p \in B(p_0, C)$, we have
$$
|\nabla^k Rm^{(i)}| (p)< C(k).
$$
\end{cor}

Now we are ready to prove Theorem (\ref{key}).

\begin{proof}[Proof of Theorem (\ref{key})]
We are essentially dealing two cases. The first case is that for sufficiently large $i$, $|Rm^{(i)}(0, p_0)|$ has a uniform lower bound. The arguments of  \cite{FH} can be applied to rule out this case. The second case is that there is a sequence of $i$ such that  $|Rm^{(i)}(0, p_0)| \rightarrow 0$. Then after an affine transformation, we can make $(D^2 u^{(i)})(p_0) = I_2$. Again we apply the arguments in \cite{FH} to get a limit metric in $\mathbb{R}^2 \times \mathbb{T}^2$. The contradiction comes from the fact that the limit metric is a flat metric but $Q^{(\infty)}(p_0) = 1$.
\end{proof}

\section{Main Theorem}
We prove our main theorem, Theorem (\ref{main}), in this section. It is clear that we can bound $Q$ in $P_{\epsilon_0}$. In fact, by modifying $Q$, we can bound $|\nabla^k Rm|$ in $P_{\epsilon_0}$ for any $k$. Since $u(t)$ satisfies the $M$-condition, we conclude that $(D^2 u_t) < C I_2$ in $P_{\epsilon_0}$. By Lemma (\ref{comparison}) and Corollary (\ref{interior control}), we obtain $(D^2 u_t) > C I_2$ in $P_{\epsilon_0}$. Similar arguments in Section 5 of \cite{FH} provide us the higher regularity of $u_t$.

\end{document}